\documentclass[10pt]{amsart}
\usepackage[english]{babel}
\usepackage{amssymb}
\usepackage{amsmath}
\usepackage{srcltx}
\usepackage{lscape}

\newtheorem{dummy}{Dummy}

\newtheorem{theorem}[dummy]{Theorem}
\newtheorem{proposition}[dummy]{Proposition}
\newtheorem{corollary}[dummy]{Corollary}

\theoremstyle{definition}
\newtheorem{definition}{Definition}

\newtheorem{example}[dummy]{Example}

\newcommand{\ignore}[1]{}

\author{C. Brown}
\author{S. Pumpl\"un}
\email{christian\_jb@hotmail.co.uk; susanne.pumpluen@nottingham.ac.uk}
\address{School of Mathematical Sciences\\
University of Nottingham\\
University Park\\
Nottingham NG7 2RD\\
United Kingdom
}

\keywords{Skew polynomial, Ore polynomial,  cyclic algebra, cyclic extension.}

\subjclass[2010]{Primary: 17A35; Secondary: 17A60, 16S36}

\begin{document}

\title[Nonassociative cyclic extensions]
{Nonassociative cyclic extensions  of fields and central simple algebras}

\begin{abstract}
We define nonassociative cyclic extensions of degree $m$ of both fields and central simple algebras over fields.
If a suitable field contains a primitive  $m$th (resp., $q$th) root of unity,  we show that
suitable nonassociative generalized cyclic division algebras yield nonassociative cyclic extensions of degree
$m$ (resp., $qs$).
Some of Amitsur's classical results on non-commutative associative  cyclic extensions
of both fields and central simple algebras are obtained as special cases.
\end{abstract}

\maketitle

%
\section*{Introduction}
%

Analogously as both for commutative field extensions \cite{AS, A2, A3, Wi} and for associative cyclic
extensions of fields and central simple algebras \cite{Am2},
 nonassocative cyclic extensions of degree $m$ of a field or a central division algebra are investigated separately for
prime characteristics and for the case that the characteristic is zero or a prime $p$ with
$gcd(p,m)=1$. Nonassociative cyclic extensions of degree $p$ in characteristic $p$ were already studied in \cite{P17}.

Let $D$ be a finite-dimensional central division algebra over a field $K$.
An (associative) central division algebra $A$ over a field $F$ is called a non-commutative \emph{cyclic extension of degree} $m$ of
$D$ over $K$, if ${\rm Aut}_F(A)$
has a cyclic subgroup of automorphisms of order $m$ which are all extended from $id_D$, and if $A$ is a free left
 $D$-module of rank $m$ \cite{Am2}. For instance, if $F$ contains a primitive $m$th root of unity, then
  generalized cyclic algebras $(D,\sigma,a)$ are cyclic extensions of $D$ of degree
$m$ \cite[Theorem 6]{Am2}. We recall that a  \emph{generalized cyclic  algebra} $(D,\sigma,a)$ is  a quotient algebra
$D[t;\sigma]/(t^m-a)D[t;\sigma]$, where $D[t;\sigma]$ is a twisted polynomial ring,
 $\sigma\in {\rm Aut}(D)$ is an automorphism such that
$\sigma|_{K}$ has finite order $m$, $F_0= {\rm Fix}(\sigma) \cap K$, and $f(t)=t^m-a\in D[t;\sigma]$ with
$d\in F_0^\times$. The special case where $D=F$ and
$F_0= {\rm Fix}(\sigma)$ yields the cyclic algebra $(F/F_0,\sigma,a)$ \cite[p.~19]{J96}.

 A finite-dimensional central simple algebra $A$ over $F$ is called a \emph{$G$-crossed product}
 if it contains a maximal field extension $K/F$ which is Galois with Galois group $G={\rm Gal}(K/F)$.
 If $G$ is solvable then $A$ is called a  \emph{solvable $G$-crossed product}.
 In \cite{BP2} we revisited a result by Albert  \cite{A} on solvable crossed products
 and gave a proof for Albert's result using generalized cyclic algebras following Petit's approach \cite{P66},
 proving that a $G$-crossed product is solvable if and only if it can be constructed as a
 chain of generalized cyclic algebras.
Hence any solvable $G$-crossed product division algebra is always
a generalized cyclic  division algebra. In particular,  hence if $F$ contains a primitive $m$th root of unity, solvable
crossed product division algebras over $F$ are non-commutative cyclic extensions.

 A generalization of associative cyclic extensions of  simple rings instead of division rings
 was considered in \cite{K}.

In this paper, we define and investigate nonassociative cyclic extensions of degree $m$ of both fields and central simple algebras
 employing nonassociative generalized cyclic division algebras: Let $A$ be a unital nonassociative division algebra.
 Then $A$ is called a \emph{nonassociative cyclic extension of $D$ of degree $m$}, if $A$ is a free
 left $D$-module of rank $m$ and ${\rm Aut}(A)$ has a cyclic subgroup $G$ of order $m$, such that for all $H\in G$,
 $H|_D=id_D$.

We show that if $F$  contains a primitive $m$th root of unity  (i.e., $F$ has characteristic 0 or characteristic $p$ with $gcd(m,p)=1$),  then
the nonassociative generalized cyclic division algebras $(D,\sigma,a)=D[t;\sigma]/(t^m-a)D[t;\sigma]$ with $a\in D^\times$ are nonassociative
 cyclic extensions of $D$ of degree $m$. Additionally,
 the subgroup of order $m$ in ${\rm Aut}_{F_0}(D,\sigma,a)$ that consists of automorphisms extending $id_D$
 contains only inner automorphisms (Corollary \ref{cor:main}). We also investigate the structure of the automorphism groups
 of nonassociative generalized cyclic algebras in general.

Note that nonassociative cyclic division algebras $(K/F,\sigma,a)$ are a special case of
 nonassociative generalized cyclic division algebras. If $F$ contains a primitive $m$th root of unity
 the nonassociative cyclic division algebras $(K/F,\sigma,a)$ are nonassociative cyclic extensions of $K$ of degree $m$.
The subgroup   of the automorphisms extending $id_K$ has order $m$, is isomorphic to ${\rm ker}(N_{K/F})$,
 and contains only inner automorphisms.  If $F$ has no non-trivial $m$th root of unity and
 $a \in K^{\times}$ is not contained in any proper subfield of $K$,
all automorphisms of $(K/F,\sigma,a)$ are inner and leave $id_K$ fixed (Theorem \ref{thm:inner}).

We point out that nonassociative generalized cyclic algebras have been recently successfully
used both in constructing space-time block codes and linear codes \cite{P15, PS15.4, P13.2, PS15}.

The paper is organized as follows: After introducing the basic terminology in Section \ref{sec:prel}, we
define nonassociative cyclic extensions and nonassociative generalized cyclic algebras in Section \ref{sec:nonass}
and investigate nonassociative cyclic extensions of a field.
In Section \ref{sec:gencyclic} we show when generalized cyclic division algebras $(D,\sigma,d)$ are
nonassociative cyclic extensions of $D$ of degree $m$.
We briefly look at the question when
 a nonassociative overring is a nonassociative cyclic extension of a field or a central simple algebra in Section
  \ref{sec:overring}.
\\\\
The results presented in this paper complements the ones for the nonassociative algebras
$(K,\delta,d)=K[t;\delta]/K[t;\delta] f(t)$ for  $f(t)=t^p-t-d\in K[t;\delta]$ constructed using a field $K$
 of characteristic $p$ together with
some derivation $\delta$ with minimum polynomial $g(t)=t^p-t\in F[t]$, $F={\rm Const}(\delta)$, and of
 the nonassociative algebras $(D,\delta,d)=D[t;\delta]/D[t;\delta]f(t)$
for $f(t)=t^p-t-d\in D[t;\delta]$ constructed using a division algebra $D$, where the derivation $\delta$ has minimum polynomial
$g(t)=t^p-t\in F[t]$ and the field $F$ characteristic $p$.
\cite{P17}.

%
%

\section{Preliminaries} \label{sec:prel}


\subsection{Nonassociative algebras} \label{subsec:nonassalgs}


Let $F$ be a field and let $A$ be an $F$-vector space. $A$ is an
\emph{algebra} over $F$ if there exists an $F$-bilinear map $A\times
A\to A$, $(x,y) \mapsto x \cdot y$, denoted simply by juxtaposition
$xy$, the  \emph{multiplication} of $A$. An algebra $A$ is called
\emph{unital} if there is an element in $A$, denoted by 1, such that
$1x=x1=x$ for all $x\in A$. We will only consider unital algebras without saying so explicitly.

The {\it associator} of $A$ is given by $[x, y, z] =
(xy) z - x (yz)$. The {\it left nucleus} of $A$ is defined as ${\rm
Nuc}_l(A) = \{ x \in A \, \vert \, [x, A, A]  = 0 \}$, the {\it
middle nucleus} of $A$ is ${\rm Nuc}_m(A) = \{ x \in A \, \vert \,
[A, x, A]  = 0 \}$ and  the {\it right nucleus} of $A$ is
${\rm Nuc}_r(A) = \{ x \in A \, \vert \, [A,A, x]  = 0 \}$. ${\rm Nuc}_l(A)$, ${\rm Nuc}_m(A)$, and ${\rm Nuc}_r(A)$ are associative
subalgebras of $A$. Their intersection
 ${\rm Nuc}(A) = \{ x \in A \, \vert \, [x, A, A] = [A, x, A] = [A,A, x] = 0 \}$ is the {\it nucleus} of $A$.
${\rm Nuc}(A)$ is an associative subalgebra of $A$ containing $F1$
and $x(yz) = (xy) z$ whenever one of the elements $x, y, z$ lies in
${\rm Nuc}(A)$.   The
 {\it center} of $A$ is ${\rm C}(A)=\{x\in A\,|\, x\in \text{Nuc}(A) \text{ and }xy=yx \text{ for all }y\in A\}$.

An algebra $A\not=0$ is called a \emph{division algebra} if for any
$a\in A$, $a\not=0$, the left multiplication  with $a$, $L_a(x)=ax$,
and the right multiplication with $a$, $R_a(x)=xa$, are bijective.
 If $A$ has finite dimension over $F$, $A$ is a division algebra if
and only if $A$ has no zero divisors \cite[ pp. 15, 16]{Sch}.
An element $0\not=a\in A$ has a \emph{left inverse} $a_l\in A$, if
$R_a(a_l)=a_l a=1$, and a \emph{right inverse}
 $a_r\in A$, if $L_a(a_r)=a a_r=1$. If  $m_r=m_l$ then we denote this element by
$m^{-1}$.

An automorphism $G\in {\rm Aut}_F(A)$ is an \emph{inner automorphism}
if there is an element $m\in A$ with left inverse $m_l$ such
that $G(x) = (m_lx)m$ for all $x\in A$. We denote such an automorphism by $G_m$.
The set of inner automorphisms $\{G_m\,|\, m\in {\rm Nuc}(A) \text{ invertible} \}$ is a subgroup of ${\rm Aut}_F(A)$.
Note that if the nucleus of $A$ is commutative, then for all $0\not=n\in
{\rm Nuc}(A)$, $G_n(x)=(n^{-1}x)n=n^{-1}x n$ is an inner automorphism of $A$
such that ${G_n}|_{{\rm Nuc}(A)}=id_{{\rm Nuc}(A)}$.


\subsection{Division algebras obtained from twisted polynomial rings}


Let $D$ be a unital division ring and $\sigma$ a ring automorphism of $D$. The \emph{twisted polynomial ring} $D[t;\sigma]$
is the set of polynomials $a_0+a_1t+\dots +a_nt^n$
with $a_i\in D$, where addition is defined term-wise and multiplication by
$ta=\sigma(a)t \quad (a\in D)$ \cite{O}. That means,
$at^nbt^m= a \sigma^{n}(b)t^{n+m} \text{ and } t^na=\sigma^n(a)t^n$
for all $a,b\in D$ \cite[p.~2]{J96}.
$R=D[t;\sigma]$ is a left principal ideal domain  and
there is a right division algorithm in $R$, i.e. for all $g,f\in R$, $f\not=0$, there exist unique $r,q\in R$
  such that ${\rm deg}(r)<{\rm deg}(f)$ and $g=qf+r$ \cite[p.~3]{J96}.
(Our terminology is the one used by Petit \cite{P66} and
 different from Jacobson's \cite{J96}, who calls what we call right a left division algorithm and vice versa.)

 An element $f\in R$ is \emph{irreducible} in $R$ if  it is no unit and it has no proper factors, i.e there do not exist
 $g,h\in R$
 such that $f=gh$ \cite[p.~11]{J96}.

 Let $f\in D[t;\sigma]$ be of degree $m$ and let ${\rm mod}_r f$ denote the remainder of right division by $f$.
 Then the vector space $R_m=\{g\in D[t;\sigma]\,|\, {\rm deg}(g)<m\}$ together with the multiplication
 $$g\circ h=gh \,\,{\rm mod}_r f $$
 becomes a unital nonassociative algebra $S_f=(R_m,\circ)$ over $F_0=\{z\in D\,|\, zh=hz \text{ for all } h\in S_f\}$
 (cf. \cite[(7)]{P66}), and $F_0$  is a subfield of $D$.
 We also denote this algebra $R/Rf$.

We note that when ${\rm deg}(g){\rm deg}(h)<m$, the multiplication of $g$ and $h$ in $S_f$ is the same as the multiplication
of $g$ and $h$ in $R$ \cite[(10)]{P66}.

A twisted polynomial $f\in R$ is \emph{right-invariant} if  $fR\subset Rf$.
  If $f$ is right invariant then $Rf$ is a two-sided ideal and conversely, every two-sided ideal in $R$ arises this way.

$S_f$ is associative if and only if $f$ is right-invariant.
In that case, $S_f$ is the usual quotient algebra  $D[t;\delta]/(f)$ \cite[(9)]{P66}.

%
%

\section{Nonassociative generalized cyclic algebras and nonassociative cyclic extensions} \label{sec:nonass}

 In the following, let $D$ be a division algebra which is
finite-dimensional over its center $F={\rm C}(D)$  and $\sigma\in {\rm Aut}(D)$ such that
$\sigma|_{F}$ has finite order $m$  and fixed field $F_0={\rm Fix}(\sigma)\cap F$.
 Note that $F/F_0$ is automatically a cyclic Galois field extension of
 degree $m$ with $\mathrm{Gal}(F/F_0) = \langle \sigma |_{F} \rangle$.

\subsection{}
Following Jacobson \cite[p.~19]{J96},
an (associative) \emph{generalized cyclic algebra} is an associative algebra
$S_f=D[t;\sigma]/D[t;\sigma]f$ constructed using a right-invariant twisted polynomial
$$f(t)=t^m-d\in D[t;\sigma]$$
with $d\in F_0^\times$.
We write $(D,\sigma, d)$ for this algebra. If $D$ is a central simple algebra over $F$ of degree $n$, then
$(D,\sigma, d)$ is a central simple algebra over $ F_0$ of degree $mn$
and
the centralizer of $D$ in $(D,\sigma, d)$ is $F$ \cite[p.~20]{J96}.
In particular, if $D=F$, $F/F_0$ is a cyclic Galois extension of degree $m$ with Galois group generated by
$\sigma$ and $ f(t)=t^m-d\in F[t;\sigma]$, we obtain the cyclic algebra $(F/F_0,\sigma,d)$ of degree $m$.

This definition generalizes to nonassociative algebras as follows:

\begin{definition}
A \emph{nonassociative generalized  cyclic algebra of degree $mn$}  is an algebra
$S_f=D[t;\sigma]/D[t;\sigma]f$ over $F_0$
with $f(t)=t^m-d\in D[t;\sigma]$, $d\in D^\times$.  We denote this algebra by $(D,\sigma, d)$.
\end{definition}

The algebra
$A=(D, \sigma, d)$, $d\in D^\times$, has dimension $m^2n^2$ over $F_0$.
In particular, if $D=F$ and
 $F/F_0$ is a cyclic Galois extension of degree $m$ with Galois group generated by $\sigma$, then
 $(F/F_0,\sigma,d)$ is a nonassociative cyclic algebra \cite{S12}.
 $A$ is associative if and only if $d\in F_0$. If  $(D, \sigma, d)$ is not associative then
${\rm Nuc}_l(A)={\rm Nuc}_m(A)=D$ and ${\rm Nuc}_r(A)=\{g\in S_f\,|\, fg\in Rf\}.$

$(D, \sigma, d)$ is a division algebra over
$F_0$ if and only if $f(t)=t^m-d\in D[t;\sigma]$ is irreducible \cite[(7)]{P66}.
 Moreover, we know that
 $f(t)=t^2-d\in D[t;\sigma]$ is irreducible  if and only if $\sigma(z)z\not=d$ for all $z\in D$,
 $f(t)=t^3-d\in D[t;\sigma]$ is irreducible  if and only if $d\not=\sigma^2(z)\sigma(z)z$
 for all $z\in D$,  and $f(t)=t^4-d\in D[t;\sigma]$ is irreducible  if and only if
$$ \sigma^2(y)\sigma(y)y+\sigma^2(x)y+\sigma^2(y)
\sigma(x)\not=0 \text{ or } \sigma^2(x)x+\sigma^2(y)\sigma(y)x \not=d$$
for all $x,y\in D$ (cf. \cite{P66} or \cite{P15}, \cite[Theorem 3.19]{CB}, see also \cite{BP18}). More generally, if $F_0$
contains a primitive $m$th root of unity and $m$ is prime then $f(t)=t^m-d\in D[t;\sigma]$ is irreducible if and only if
$d\not=\sigma^{m-1}(z)\cdots\sigma(z)z$ for all $z\in D$ (\cite[Theorem 3.11]{CB}, see also \cite[Theorem 6]{P15}).

Amitsur's definition \cite{Am2} of cyclic extensions generalizes to the nonassociative setting as follows:

\begin{definition} Let $m\geq 2$.
Let $A$ be a nonassociative division algebra with center $F_0$ and $D$ an associative division algebra with
center $F$. Then $A$ is a \emph{nonassociative cyclic extension} of $D$ of degree $m$, if $A$ is a free
 left $D$-module of rank $m$ and ${\rm Aut}(A)$ has a cyclic subgroup $G$ of order $m$, such that for all $H\in G$,
 $H|_D=id$.
\end{definition}

\subsection{Nonassociative cyclic extensions of a field}

For a nonassociative cyclic algebra $(K/F,\sigma,d)$ of degree $m$,  and for all $k\in K$
 such that $N_{K/F}(k)=1$, the map
  $$H_{id,k}(\sum_{i=0}^{m-1}a_it^i)=  a_0 + \sum_{i=1}^{m-1} a_i \big( \prod_{l=0}^{i-1}\sigma^l(k) \big) t^i$$
 is an inner $F$-automorphism of $(K/F,\sigma,d)$ extending  $id_K$.
 The subgroup generated by the automorphisms $H_{id,k}$ is isomorphic to ${\rm ker}(N_{K/F})$ \cite[Theorem 19]{BP}.

The maps $H_{id,k}$ are the only $F$-automorphisms of $(K/F,\sigma,d)$, unless for some $j\in\{1,\dots,m-1\}$,
 $\sigma^j$ can be extended to an $F$-automorphism of $(K/F,\sigma,d)$ as well.
More precisely, the automorphism $\tau=\sigma^j$ with $j\in\{1,\dots,m-1\}$ can be extended to an $F$-automorphism $H$ of
$(K/F,\sigma,d)$, if and only if there is an element
 $k\in K$ such that
  \begin{equation}\label{condition_on_k}
  \sigma^j(d)= N_{K/F}(k)d.
  \end{equation}
   The extension then has the form $H=H_{\tau,k}$ with
 \begin{equation}\label{automorphism_of_Sf form of H}
 H_{\tau,k}(\sum_{i=0}^{m-1}a_it^i)=\tau(a_0) + \sum_{i=1}^{m-1} \tau(a_i)
\big( \prod_{l=0}^{i-1}\sigma^l(k) \big) t^i
\end{equation}
\cite[Theorem 4]{BP}. We then immediately get the following partial generalization of \cite[Theorem 6]{Am2}:

\begin{theorem} \label{thm:Nonassociative cyclic algebra is cyclic extension}
Suppose $F$ contains a primitive $m$th root of unity $\omega$, $A = (K/F,\sigma,d)$ is a nonassociative cyclic division
algebra of degree $m$ over $F$, and $d \in K \setminus F$. Then $A$ is a nonassociative cyclic extension of $K$ of degree $m$.
 The generating automorphism of the subgroup of ${\rm Aut}_F(A)$ of order $m$ is given by $H_{id, \omega}$.
\end{theorem}

\begin{proof}
 $\langle H_{\mathrm{id},\omega}\rangle$ is a cyclic subgroup of $\mathrm{Aut}_F(A)$ of order $m$ by \cite[Theorem 20]{BP}.
It consists of automorphisms extending $id_K$, therefore $A$ is a nonassociative cyclic extension of $K$.
\end{proof}

\begin{corollary}
If $m$ is prime, $F$ contains a primitive $m$th root of unity and $K/F$ is a cyclic Galois extension of degree $m$,
 then $K$ has a nonassociative cyclic extension of degree $m$.
\end{corollary}

\begin{proof}
Let $d \in K \setminus F$ and suppose $\mathrm{Gal}(K/F) = \langle\sigma \rangle$. Then since $m$ is prime, the nonassociative
cyclic algebra $A = (K/F,\sigma,d)$ is a division algebra \cite[Corollary 4.5]{S12}. Thus $A$ is a nonassociative cyclic
extension of $K$ by Theorem \ref{thm:Nonassociative cyclic algebra is cyclic extension}.
\end{proof}

If  $F$ has no non-trivial $m$th root of unity, we obtain:

\begin{theorem}\label{thm:inner}
Suppose $F$ has no non-trivial $m$th root of unity.
Let $A = (K/F,\sigma,d)$ be a nonassociative cyclic algebra of degree $m$ where $d \in K^{\times}$ is not contained
in any proper subfield of $K$.
Then every $F$-automorphism of $A$ leaves $K$ fixed and
$$\mathrm{Aut}_F(A) \cong \mathrm{ker}(N_{K/F}).$$
In particular, all automorphisms of $A$ are inner.
\end{theorem}

\begin{proof}
Every automorphism of $A$ has the form $H_{id,k}$:
suppose that there exist $j \in \{ 1, \ldots, m-1 \}$ and $k \in K^{\times}$ such that
$H_{\sigma^j,k} \in \mathrm{Aut}_{F}(A)$. This implies $H_{\sigma^j,k}^2 = H_{\sigma^j,k} \circ H_{\sigma^j,k} \in
\mathrm{Aut}_{F}(A)$  and
\begin{equation} \label{eqn:Aut(S_f) F does not contain primitive root of unity 1}
\begin{split}
H_{\sigma^j,k}^2 & \Big( \sum_{i=0}^{m-1} x_i t^i \Big) = \sigma^{2j}(x_0)  + \sum_{i=1}^{m-1} \sigma^{2j}(x_i) \Big( \prod_{q=0}^{i-1} \sigma^{j+q}(k) \sigma^q(k) \Big) t^i.
\end{split}
\end{equation}
Now $H_{\sigma^j,k}^2$ must have the form $H_{\sigma^{2j},l}$ for some $l \in K^{\times}$, and comparing
\eqref{automorphism_of_Sf form of H} and \eqref{eqn:Aut(S_f) F does not contain primitive root of unity 1} yields
$l = k \sigma^j(k)$. Similarly, $H_{\sigma^j,k}^3 = H_{\sigma^{3j},s} \in \mathrm{Aut}_{F}(A)$ where
$s = k \sigma^j(k) \sigma^{2j}(k)$.
Continuing in this manner we conclude that the automorphisms $H_{\sigma^j,k}, H_{\sigma^{2j},l}, H_{\sigma^{3j},s},
\ldots$ all satisfy  (\ref{condition_on_k}) implying that
\begin{align} \label{eqn:Aut(S_f) F does not contain primitive root of unity 2}
\begin{split}
\sigma^j(d) &= N_{K/F}(k) d, \\
\sigma^{2j}(d) &= N_{K/F}(k \sigma^j(k)) d = N_{K/F}(k)^2 d, \\
\vdots & \qquad \qquad \vdots \\
d = \sigma^{n j}(d) &= N_{K/F}(k)^{n} d,
\end{split}
\end{align}
where $n = m/\mathrm{gcd}(j,m)$ is the order of $\sigma^j$.
Note that $\sigma^{ij}(d) \neq d$ for all $i \in \{ 1, \ldots, n-1 \}$ since $d$ is not contained in any proper
subfield of $K$. Therefore $N_{K/F}(k)^{n} = 1$ and $N_{K/F}(k)^i \neq 1$ for all $i \in \{1, \ldots, n - 1 \}$
by \eqref{eqn:Aut(S_f) F does not contain primitive root of unity 2}, i.e. $N_{K/F}(k)$ is a primitive $n$th
 root of unity, thus also an $m$th root of unity, a contradiction. This proves the assertion.
\end{proof}

Note that if  $d \in K^{\times}$ is not contained
in any proper subfield of $K$ then $1,d,\dots,d^{m-1}$ are linearly independent over $F$ and thus
$A$ is a division algebra \cite{S12}. In particular, if $m$ is prime then $1,d,\dots,d^{m-1}$ are
linearly independent over $F$. This yields for a field $F$ of arbitrary  characteristic:

\begin{corollary}
Suppose that $F$ has no non-trivial $m$th root of unity.
If $d \in K^{\times}$ is not contained in any proper subfield of $K$ (e.g. if $m$ is prime),
  and ${\rm ker}(N_{K/F})$ has a subgroup of order $m$,
then any cyclic algebra $A=(K/F, \sigma, d)$ is a cyclic extension of $K$ of degree $m$.
\end{corollary}

\begin{example}
Let $K = \mathbb{F}_{q^m}$
be a finite field,  $q = p^r$ for some prime $p$, $\sigma$ an automorphism of $K$ of order $m\geq 2$ and
$F={\rm Fix}(\sigma)=\mathbb{F}_{q},$
 i.e. $K/F$ is a cyclic Galois extension of degree $m$
with ${\rm Gal}(K/F)=\langle\sigma\rangle.$  Then ${\rm ker}(N_{K/F})$ is a cyclic group of order
$s=(q^m-1)/(q-1)$ and any division algebra $(K/F,\sigma,d)$ has exactly $s$ inner
automorphisms, all of them extending $id_K$.
 The subgroup they generate is cyclic and isomorphic to ${\rm ker}(N_{K/F})$  \cite{BP3}.
  Hence if $m$ divides $s$, which is the case if $F$ contains a primitive $m$th root of unity, then there is a subgroup of automorphisms of order $m$ extending $id_K$ and hence
  $(K/F,\sigma,d)$  is a cyclic extension of $K$ of degree $m$.
\end{example}

%
%

\section{Nonassociative cyclic extensions of a central simple algebra} \label{sec:gencyclic}

\subsection{}
From now  until stated otherwise, let $A = (D, \sigma, d)$ be a nonassociative generalized cyclic algebra of degree $mn$  over $F_0$,
 for some  $d \in  D \setminus F_0$. We first determine the automorphisms of $A$:

\begin{theorem} \label{thm:aut2}
 (i) Suppose $\tau\in {\rm Aut}_{F_0}(D)$ commutes with $\sigma$. Then
$\tau$ can be extended to an automorphism $H \in {\rm Aut}_{F_0}(A)$, if and only if there is some
$k \in F^\times$ such that $\tau(d) = N_{F/F_0}(k) d.$
 In that case, the extension $H$ of $\tau$ has the form $H=H_{\tau,k}$ with
$$H_{\tau,k}( \sum_{i=0}^{m-1} a_i t^i )=\tau(a_0)+ \sum_{i=1}^{m-1}\tau(a_i) \big( \prod_{l=0}^{i-1} \sigma^l(k) \big) t^i.$$
All maps $H_{\tau,k}$ where $\tau \in {\rm Aut}_{F_0}(D)$ commutes with $\sigma$ and where $k \in F^\times$  such that
 $\tau(d) = N_{F/F_0}(k) d$ (hence $N_{F/F_0}(k)^{mn}=1$), are automorphisms of $A$.

 In particular, for $\tau \not= id$ and $d\not\in {\rm Fix}(\tau)$, $N_{F/F_0}(k) \not= 1$.
\\ (ii) $id\in {\rm Aut}(D)$ can be extended to an automorphism $H\in {\rm Aut}_{F_0}(A)$, if and only if there is some
$k\in F^\times$ such that
$N_{F/F_0}(k) = 1.$
 In that case, the extension $H$ of $id$ has the form $H=H_{id,k}$ with
$$H_{id,k}( \sum_{i=0}^{m-1}a_i t^i)= a_0+ \sum_{i=1}^{m-1} a_i \big( \prod_{l=0}^{i-1} \sigma^l(k) \big) t^i.$$
All $H_{id,k}$ where  $k \in F^{\times}$  such that $N_{F/F_0}(k) = 1$
 are automorphisms of $A$.
\end{theorem}

\begin{proof}
(i) Let $H \in {\rm Aut}_{F_0}(A)$, then $H|_D\in {\rm Aut}_{F_0}(D)$, since $H$ leaves the left nucleus invariant.
Thus $H|_D = \tau$ for some $\tau\in {\rm Aut}_{F_0}(D)$. Write
$H(t) = \sum_{i=0}^{m-1} k_i t^i$ for some $k_i \in D$, then we have
$$H(tz) = H(t)H(z) = \big( \sum_{i=0}^{m-1} k_i t^i \big) \tau(z) = \sum_{i=0}^{m-1} k_i \sigma^i(\tau(z)) t^i,$$
and
$$H(tz) = H(\sigma(z)t) = \tau(\sigma(z)) \sum_{i=0}^{m-1} k_i t^i = \sum_{i=0}^{m-1} \tau(\sigma(z)) k_i t^i$$
for all $z\in D$. Comparing the coefficients of $t^i$ yields
$$k_i \sigma^i(\tau(z)) = k_i \tau(\sigma^i(z)) = \tau(\sigma(z)) k_i \text{ for all } i =\{0,\dots,m-1\}$$
for all $z \in D$ since $\sigma$ and $\tau$ commute. In particular, we obtain
$$k_i (\tau(\sigma^i(z)) - \tau(\sigma(z))) = 0 \text{ for all } i  \in \{0,\dots,m-1\}$$
for all $z \in F$, i.e. $k_i = 0$ or $\sigma \vert_F = \sigma^i \vert_F$ for all $i  \in \{0, \dots, m-1\}$. As $\sigma \vert_F$ has order $m$, this means $k_i = 0$ for all $1 \neq i \in \{ 0, \ldots, m-1 \}$.
 For $i = 1$, this yields $k_1 \tau(\sigma(z)) = \tau(\sigma(z)) k_1$
for all $z\in D$, hence $k_1 \in F$. This implies $H(t) = kt$ for some $k \in F^{\times}$.

Since
$$H(z t^i) = H(z) H(t)^i = \tau(z) (kt)^i = \tau(z) \Big( \prod_{l=0}^{i-1} \sigma^l(k) \Big) t^i,$$
for all $i \in \{ 1, \ldots, m-1 \}$ and all $z \in D$, $H$ has the form
$$H_{\tau,k}: \sum_{i=0}^{m-1} a_i t^i \mapsto \tau(a_0) + \sum_{i=1}^{m-1} \tau(a_i) \big( \prod_{l=0}^{i-1} \sigma^l(k) \big) t^i,$$
for some $k \in F^{\times}$.

Comparing the constant terms in $H(t)^m = H(t^m) = H(d)$ implies
$$\tau(d) = k \sigma(k) \cdots \sigma^{m-1}(k) d = N_{F/F_0}(k) d.$$
Let $N = N_{F/F_0} \circ N_{D/F}$ be the norm  of the $F_0$-algebra $D$. Applying $N$ to both sides of the equation
 yields $N(d) = N(k)^m N(d)$, so that $N(k)^m = 1$. Now $k \in F^\times$ and $D$ has degree $n$, thus
$$N(k) = N_{F/F_0}( N_{D/F}(k)) = N_{F/F_0}(k^n) = N_{F/F_0}(k)^n,$$
and so $N(k)^m = N_{F/F_0}(k)^{nm} = 1.$

Finally, the fact that the maps $H_{\tau,k}$ are automorphisms when $\tau$ commutes with $\sigma,$ and
$\tau(d)=N_{F/F_0}(k)d,$ can be shown similarly to the proof of \cite[Theorem 4]{BP}, see also
\cite{CB}.
\\ (ii) In particular, for $\tau=id$, we get from (i) that $H$  has the form
$$H_{id,k}: \sum_{i=0}^{m-1}a_i t^i \mapsto a_0+ \sum_{i=1}^{m-1} a_i \big( \prod_{l=0}^{i-1} \sigma^l(k) \big) t^i$$
for some $k\in F^\times$ with $k\sigma(k)\cdots \sigma^{m-1}(k) = N_{F/F_0}(k) = 1.$
\end{proof}

 The above is proved for a more general set-up  in the first author's PhD thesis \cite{CB}.
 Note that the automorphisms $H_{\tau,k}$ are restrictions of automorphisms of the twisted polynomial ring $D[t;\sigma]$.

 \begin{corollary}\label{cor:7}
 (i) The subgroup of $F_0$-automorphisms of  $A$
extending $id_D\in {\rm Aut}_{F_0}(D)$ is isomorphic to
$$\{ k \in F^\times\,|\, k\sigma(k)\cdots \sigma^{m-1}(k) = 1 \}.$$
 (ii)
 If $F_0$ contains a primitive $m$th root of unity $\omega$, then $\langle H_{id,\omega}\rangle$
is a cyclic subgroup of ${\rm Aut}_{F_0}(A)$ of order $m$.
 \end{corollary}

\subsection{}  We obtain the following generalization of \cite[Theorem 6]{Am2}:

\begin{corollary}
Suppose $F_0$ contains a primitive $m$th root of unity.
If $f(t) = t^m-d \in D[t;\sigma]$ is irreducible, then $A$
 is a nonassociative cyclic extension of $D$ of degree $m$. In particular, if $m$ is prime and
$$d\not=\sigma^{m-1}(z)\cdots\sigma(z)z$$
 for all $z\in D$, then $A$ is a nonassociative cyclic extension of $D$ of degree $m$.
\end{corollary}

\begin{proof}
If $F_0$ contains a primitive $m$th root of unity $\omega$, then $\langle H_{id,\omega}\rangle$
is a cyclic subgroup of ${\rm Aut}_{F_0}(A)$ of order $m$ by Corollary \ref{cor:7} (ii).
If $m$ is prime, then
$f(t)=t^m-d\in D[t;\sigma]$ is irreducible if and only if
$$d\not=\sigma^{m-1}(z)\cdots\sigma(z)z$$
 for all $z\in D$.  The rest is trivial.
\end{proof}

\begin{proposition}
 Every automorphism $H_{id,k}$ of $A$ is an inner automorphism
$$G_c(\sum_{i=0}^{m-1}a_it^i)=(c^{-1}\sum_{i=0}^{m-1}a_it^i) c$$
for some  $c\in F^\times$ satisfying $k=\sigma(c)c^{-1}$.
\end{proposition}

\begin{proof}
For all $k\in F$ such that $ N_{F/F_0}(k)=1$, $H_{id,k}$
 is an $F$-automorphism extending  $id_D$. These are the only $F_0$-automorphisms of $A$, unless
 $\tau\not=id$ can be also extended.
By Hilbert's Satz 90, $N_{F/F_0}(k)=1$ if and only if there is $c\in F^\times$ such that $k=c^{-1}\sigma(c)$
 \cite{LL}.
So there is $c\in F^\times$ such that $k=c^{-1}\sigma(c)$ and
$$k\sigma(k)\cdots \sigma^{i-1}(k)=c\sigma^i(c),\quad i=1.\dots,m-1$$
yields that $H_{id,k}=G$ with
$$G(\sum_{i=0}^{m-1}a_it^i)=
a_0+a_1c^{-1}\sigma(c)t+\sum_{i=2}^{m-1} a_i c^{-1} \sigma^i(c) t^i,$$
which is an inner automorphism, since $G=G_c$ with
$$G_c(\sum_{i=0}^{m-1}a_it^i)=(c^{-1}\sum_{i=0}^{m-1}a_it^i) c.$$
Note that here we use that $F=C(D)$.
\end{proof}

 \begin{corollary} \label{cor:main}
 If $F_0$ contains a primitive $m$th root of unity $\omega$, then $A$ is a cyclic extension of $D$ of order $m$,
 and all automorphisms extending $id_D$ are inner.
 \end{corollary}

\begin{example} \label{ex:code}
Let $F$ and $L$ be fields and let $K$ be a cyclic Galois extension of both $F$ and $L$ such that
 $[K:F] = n$, $[K:L] = m$,  ${\rm Gal}(K/F) = \langle \gamma \rangle$  and ${\rm Gal}(K/L) = \langle  \sigma \rangle$,
  and  $\sigma \circ\gamma = \gamma \circ\sigma$. Define $F_0=F \cap L$.

Let $D=(K/F, \gamma, c)$ be a cyclic division algebra of degree $n$ with $c \in F_0$, i.e.
$D\cong D_0\otimes_{F_0} K$ for some cyclic algebra $D_0=(F/F_0, \gamma, c)$.
Let $1,e.\dots,e^{n-1}$ be the canonical basis of $D$, that is $e^n=c,$ $ex=\gamma(x)e$ for every $x\in K$.
 For $x= x_0 + x_1e + x_2e^2 +\dots + x_{n-1}e^{n-1}\in D$, define an $L$-linear map
 $\sigma\in {\rm Aut}_L(D)$  via
 $$\sigma(x)=\sigma(x_0) +  \sigma(x_1)e +\sigma(x_2) e^2 +\dots +\sigma(x_{n-1}) e^{n-1}$$
(note that $c\in L$ implies  $\sigma(xy) = \sigma(x)\sigma(y)$ for all $x,y \in D$). Then
$\sigma\in {\rm Aut}_{F_0}(D)$ has order $m$.
For all $d \in D^\times$,
$$D[t;\sigma]/D[t;\sigma](t^m-d)=(D,\sigma,d)$$
is a generalized nonassociative cyclic algebra of degree $mn$ over $F_0$ (used for instance in \cite{PS15.4}).
$(D,\sigma,d)$ is associative if and only if $d\in F_0$. In the special case that $d \in F^\times$,
$$(D, \sigma, d)= (L/F_0, \gamma, c)\otimes_{F_0} (F/F_0, \sigma, d)$$
is the tensor product of an associative and a nonassociative cyclic algebra.

 If $F_0$ contains a primitive $m$th root of unity  and $d\in D^\times\setminus F_0$ is chosen such that
 $f(t)=t^m-d\in D[t;\sigma]$ is irreducible, then $(D, \sigma, d)$ is a cyclic extension of $D$ of order $m$,
 and all automorphisms extending $id_D$ are inner (Corollary \ref{cor:main}).
 Recall that if $m$ is prime then $f(t)=t^m-d\in D[t;\sigma]$ is irreducible if and only if
$d\not=\sigma^{m-1}(z)\cdots\sigma(z)z$ for all $z\in D$.

For $m=2$, this algebra is studied in \cite{P13.2}, and used in the  codes constructed in \cite{MO13}.
For $d\in F^\times$ the algebra is used in~\cite{R13}, see also \cite{PS15.4}.
\end{example}

\subsection{}
In the following, let $D$ be a division algebra which is finite-dimensional over its center $F = C(D)$,
$\sigma \in \mathrm{Aut}(D)$ an automorphism such that $\sigma \vert_F$ has finite order $q$ and fixed field
$F_0 = \mathrm{Fix}(\sigma) \cap F$.  If $D$ has degree $n$ then the
 associative generalized cyclic algebra $A = (D,\sigma,a)$
 has degree $qn$ over $F_0$. We choose $a \in F_0$ such that $A$ is a division algebra.

Now assume $F_0$ contains a primitive $q$th root of unity $\omega$. Then $\tau = H_{id_D,\omega}: A \rightarrow A$
generates a cyclic subgroup of $\mathrm{Aut}_{F_0}(A)$ of order $q$ by \cite[Theorem 6]{Am2} which consists of automorphisms
which all extend $id_D$.  We obtain the following generalization of \cite[Theorem 7]{Am2}:

\begin{theorem} \label{thm:Generalized Amitsur thm 7}
Suppose there exists $\rho \in \mathrm{Aut}(A)$, $b \in A$ and $1 \neq k \in F_0$ such that
\\ (1) $\tau$ commutes with $\rho$,
\\ (2) $\tau(b) = k \rho(k) \cdots \rho^{m-1}(k) b$,
\\ (3) $k^q$ is a primitive $m$th root of unity,
\\ (4) $t^m-b \in A[t;\rho]$ is irreducible, and
\\ (5) the algebra $B = A[t;\rho]/A[t;\rho] (t^m-b)$ is either associative, or finite-dimensional over $F_0 \cap \mathrm{Fix}(\rho)$,
or finite-dimensional over $\mathrm{Nuc}_r(B)$. \\
Then $B$ is a nonassociative cyclic extension of $D$ of degree $mq$ which contains $A$.
\end{theorem}

\begin{proof}
Since $B$ is a free left $A$-module of rank $m$ and $A$ is a free left $D$-module of rank $q$,  $B$ is a free
left $D$-module of rank $mq$. Furthermore, (4) and (5) yield that $B$ is a division algebra by \cite[(7)]{P66}. Define the map
$$H_{\tau,k}: B \rightarrow B, \ \sum_{i=0}^{m-1} x_i t^i \mapsto \tau(x_0) + \sum_{i=1}^{m-1} \tau(x_i) \big(
\prod_{l=0}^{i-1} \rho^l(k) \big) t^i \qquad (x_i \in A),$$
then (1) and (2) together imply that $H_{\tau,k}$ is an automorphism of $B$ by \cite[Theorem 4]{BP}.

$H_{\tau,k}$ has order $mq$: We have $\tau(k) = k$ because $k \in F_0 \subset D$. Therefore straightforward
calculations yield $H_{\tau,k}^2 = H_{\tau,k} \circ H_{\tau,k} = H_{\tau^2,k\tau(k)} = H_{\tau^2, k^2}$, $H_{\tau,k}^3 = H_{\tau^3,k^3}$ etc., thus $H_{\tau,k}$ will have order at least $q$. After $q$ steps we obtain $H_{\tau,k}^q = H_{id_A,v}$ with $v = k^q$ and so $H_{\tau,k}$ has order $mq$ by (3).

Finally $H_{\tau,k} \vert_D = \tau \vert_D = id_D$, hence we conclude $B$ is a nonassociative cyclic extension of
$D$ of degree $mq$.
\end{proof}

%
%

\section{When is an ring a nonassociative cyclic extension?}\label{sec:overring}

A nonassociative ring $A\not=0$  is called  a \emph{right division ring}, if for all $a\in A$, $a\not=0$,
the right multiplication  with $a$, $R_a(x)=xa$, is bijective.
If $D$ is a division ring and $f$ is irreducible, then $S_f=D[t;\sigma]/D[t;\sigma]f$ is
a right division algebra and has no zero divisors (\cite[(6)]{P66} or  \cite{BP18}).

\begin{theorem} (cf. \cite[(3), (6)]{P66})\label{thm:(3)}
\\ (i) Let $S$ be a nonassociative ring with multiplication  $\circ$. Suppose that
\\ (1)  $S$ has an associative subring $D$ which is a division algebra
and $S$ is a free left $D$-module of rank $m$, and there is $t\in S$ such that
 $t^j$, $0\leq i<m$ is a basis of $S$ over $D$, when defining $t^{j+1}=t\circ t^{j}$, $t^0=1$;
\\ (2) for all $a\in D$, $a\not=0$, there are $a_1, a_2\in D$, $a_1\not=0$, such that $t\circ a=a_1\circ t+a_2$;
\\ (3) for all $a,b,c\in D$, $i+j<m$, $k<m$, we have
$[a\circ t^i,b\circ t^j,c\circ t^k]=0.$
\\
 Then $S\cong S_f$ with $f(t)\in D[t;\sigma,\delta]$ and $\sigma$, $\delta$ defined via
$t\circ a=\sigma(a)\circ t+\delta(a)$ and where the polynomial
$f(t)=t^m-\sum_{i=0}^{m-1}d_it^i$ is given by
$t^m=\sum_{i=0}^{m-1}d_it^i$ with $t^0=1,$ $t^{i+1}=t\circ t^i$, $0\leq i<m$.
\\ (ii) If $S$ is a right division ring in (i) then $f$ is irreducible.
\end{theorem}

 Theorem \ref{thm:(3)}  yields the nonassociative analogues to the existence conditions for associative cyclic
 extensions in \cite[Theorem 6]{Am2}.

\begin{theorem} \label{thm:class2}
(i) Let $S$ be a nonassociative ring with multiplication  $\circ$, which has a field $K$ as a subring,
and is a free left $K$-vector space of dimension $m$. Suppose that
\\ (1) there is $t\in S$ such that
 $t^i$, $0\leq i<m$, is a basis of $S$ over $K$ when defining $t^0=1,$ $t^{i+1}=t\circ t^i$, $0\leq i<m$;
\\ (2) for all $a\in K$, $a\not=0$, there is $a'\in K^\times$, such that $t\circ a=a'\circ t$;
\\ (3) for all $a,b,c\in K$, $i+j<m$, $k<m$, we have
$[a\circ t^i,b\circ t^j,c\circ t^k]=0;$
\\ (4)  $t^m=d$ for some $d\in K^\times$;
\\ (5) the map $\sigma:K\to K$, $\sigma(a)=a'$,
 has order $m$ and fixed field $F=\{a\in K\,|\, t\circ a=a\circ t\}$  containing a primitive $m$th root of unity $\omega$,
and $K/F$ is a finite cyclic Galois extension.
\\
 Then $S\cong S_f=(K/F,\sigma,d)$ with $f(t)=t^m-d\in K[t;\sigma]$.
\\ (ii) If $S$ is a right division ring in (i) then $f$ is irreducible and $S\cong (K/F,\sigma,d)$ is a
nonassociative cyclic extension of $K$ of degree $m$.
\end{theorem}

\begin{proof}
(1), (2) and (3) imply that $S\cong S_f$ with $f\in K[t;\sigma]$ and $\sigma$ defined via
$t\circ a=\sigma(a)\circ t,$ i.e. $\sigma(a)=a',$
and where the polynomial $f(t)=t^m-\sum_{i=0}^{m-1}d_it^i$ is given by $t^m=\sum_{i=0}^{m-1}d_it^i$
for some suitably chosen $d_i$ (cf. \cite[(3)]{P66}).
(4) implies that indeed $f(t)=t^m-d$. (5) guarantees that
$(K/F,\sigma,d)$ where  $F$ contains a primitive $m$th root of unity $\omega$.
\\ (ii) Here we are in the setup of
Theorem \ref{thm:Nonassociative cyclic algebra is cyclic extension} which  yields the assertion:
 $F$ contains a primitive $m$th root of unity $\omega$, so $\langle H_{id,\omega}\rangle$
is a cyclic subgroup  of order $m$ of the division algebra $(K/F,\sigma,d)$.
\end{proof}

For nonassociative cyclic extensions of a central simple algebra $D$
we obtain from Theorem \ref{thm:(3)}:

\begin{theorem} \label{thm:class3}
 (i) Let $S$ be a nonassociative ring with multiplication  $\circ$, which has an associative
  subring $D$ which is a division algebra and $S$ is a free left $D$-module of rank $m$. Suppose that
\\ (1)  there is $t\in S$ such that
 $t^i$, $0\leq i<m$, is a basis of $S$ over $D$ when defining $t^0=1,$ $t^{i+1}=t\circ t^i$, $0\leq i<m$;
\\ (2) for all $a\in D$, $a\not=0$, there are $a'\in D$, $a'\not=0$, such that $t\circ a=a'\circ t$;
\\ (3) for all $a,b,c\in D$, $i+j<m$, $k<m$, we have $[a\circ t^i,b\circ t^j,c\circ t^k]=0$;
\\ (4)  $t^m=d$;
\\ (5) the map $\sigma:D\to D$, $\sigma(a)=a'$, has order $m$, fixed field $\{a\in D\,|\, t\circ a=a\circ t\}$
and $D/F$ is a central simple algebra, where $F_0=F\cap {\rm Fix}(\sigma)$
 contains a primitive $m$th root of unity $\omega$.
\\
 Then $S\cong S_f=(D,\sigma,d)$ with $f(t)=t^m-d\in D[t;\sigma]$.
\\ (ii) If $S$ is a right division ring and $D$ a central simple algebra in (i), then $f$ is irreducible  and $S$ a
nonassociative cyclic extension of $D$ of degree $m$.
\end{theorem}

\begin{proof}
(1), (2) and (3) imply that $S\cong S_f$ with $f\in D[t;\sigma]$ and $\sigma$ defined via
$t\circ a=\sigma(a)\circ t,$ i.e. $\sigma(a)=a',$
and where the polynomial $f(t)=t^m-\sum_{i=0}^{m-1}d_it^i$ is given by $t^m=\sum_{i=0}^{m-1}d_it^i$
for some suitably chosen $d_i$ (cf. \cite[(3)]{P66}).
(4) implies $f(t)=t^m-d$. (5) guarantees that
$S\cong (D,\sigma,d)$ where  $F$ contains a primitive $m$th root of unity $\omega$.
\\ (ii) Here we are in the setup of
Theorem \ref{thm:aut2} which  yields the assertion, since
 $F$ contains a primitive $m$th root of unity $\omega$, $\langle H_{id,\omega}\rangle$
is a cyclic subgroup  of order $m$ of the division algebra $(D,\sigma,d)$.
\end{proof}


\end{document}